\documentclass[12pt,a4paper,oneside]{amsart}

\usepackage{graphicx} 
\usepackage{float}
\usepackage[skip=1ex, 
            labelfont=bf]{caption}
\usepackage[utf8]{inputenc}
\usepackage{amsmath}
\usepackage{amsfonts}
\usepackage{amssymb}
\usepackage{hyperref}
\usepackage{framed,color}
\definecolor{shadecolor}{rgb}{1,0.8,0.3}
\usepackage{xcolor}
\usepackage{lmodern}
\usepackage{tikz}
\usetikzlibrary{decorations.pathreplacing}

\usepackage{amsthm}
\newtheorem{theorem}{Theorem}[section] 
\newtheorem{proposition}{Proposition}[section]
\newtheorem {corollary}[theorem]{Corollary}
\newtheorem{lemma}{Lemma}[section]

\theoremstyle{definition}
\newtheorem{definition}{Definition}[section]
\newtheorem{example}{Example}[section]
\newtheorem{remark}{Remark}[section]

\DeclareMathOperator{\Supp}{Supp}
\DeclareMathOperator{\Conv}{Conv}
\DeclareMathOperator{\Newton}{Newton}

\newcommand{\N}{\mathbb{N}}
\newcommand{\Z}{\mathbb{Z}}
\newcommand{\R}{\mathbb{R}}

\title{Supersymmetric Schur polynomials have saturated Newton polytopes}

\author{Dang Tuan Hiep}
\email{hiepdt@dlu.edu.vn }
\address{Faculty of Mathematics and Computer Science, Da Lat University, Building A24, No. 1, Phu Dong Thien Vuong Road, Lam Vien - Da Lat, Lam Dong, Vietnam.}

\author{Khai-Hoan Nguyen-Dang}
\email{khaihoann@gmail.com}
\address{Morningside center of mathematics, Chinese academy of sciences, No.55, Zhongguancun East Road, Beijing, 100190, China}

\date{\today}

\begin{document}

\begin{abstract}
    We prove that every supersymmetric Schur polynomial has a saturated Newton polytope (SNP). Our approach begins with a tableau-theoretic description of the support, which we encode as a polyhedron with a totally unimodular constraint matrix. The integrality of this polyhedron follows from the Hoffman–Kruskal criterion, thereby establishing the SNP property. 
\end{abstract}
\maketitle
\tableofcontents

\section{Introduction}
For a multivariate polynomial \(f=\sum_{\alpha\in\Bbb N^{d}}c_{\alpha}\,{\mathbf z}^{\alpha}\) its \emph{Newton polytope} is the convex hull \(
   \Newton(f):=\Conv\{\alpha\mid c_{\alpha}\neq0\}\subset\R^{d}.
\) We say that \(f\) has a \emph{saturated Newton polytope} (SNP for short) if every lattice point of \(\Newton(f)\) already occurs in the support, i.e., \(\Newton(f)\cap\Z^{d}=\Supp(f)\), where \(\Supp(f):=\{\alpha\mid c_{\alpha}\neq0\}.\) Equivalently, the Newton polytope can be seen as the \emph{integer hull} of its support. In recent years, the saturation property of Newton polytopes has emerged as a powerful tool, underpinning advances in commutative algebra (e.g.\ \cite{CCLMZ20,CCC23,HNP25}), algebraic geometry (e.g.\ \cite{Fei23,CLM24}), representation theory (e.g.\ \cite{FMD18,MS22,BJK24,BJK24b}), and manifold theory (e.g.\ \cite{CC24}).  Moreover, SNP plays a pivotal role in establishing a Lorentzian polynomial \cite{BH20}, thereby unlocking a host of applications across diverse areas (e.g.\ \cite{HMMS22,MMS24,ATZ24,KMD25}).

Rado’s permutahedron theorem (1952) (e.g. \cite{Ra52}) shows that every Schur polynomial \(s_{\lambda}\) has SNP because \(\Newton(s_{\lambda})\) equals the permutahedron \(P_{\lambda}\), the convex hull of the $S_d$-orbit of $\lambda$ in $\R^d$. The dominance order arguments force all interior weights to appear (see \cite{MTY19}). The authors in \cite{MTY19} also compiled a systematic list of families known (or conjectured) to have SNP (e.g., skew Schur, Stanley, Hall–Littlewood, certain resultants, discriminants, $\dots$) and crucially framed the question in a polyhedral language powerful enough to accommodate Schubert‐type bases. Since the work, a sequence of breakthroughs has confirmed SNP for most outstanding cases: \emph{symmetric Grothendieck polynomials} \cite{EY17}, \emph{key polynomials and Schubert polynomials} \cite{FMD18}, \emph{Grothendieck and dual characters} \cite{HMMS22}, \emph{double Schubert polynomials} \cite{CCMM23}, \emph{good class of symmetric polynomials} \cite{NNDD23}, \emph{chromatic symmetric functions}\cite{MMS24}, \emph{Demazure characters} \cite{BJK24}, \ldots.

Classical symmetric function theory is built on one alphabet of one set of variables and governs the characters of the general linear group~$\mathrm{GL}_{n}$. When the representation theory is upgraded from \emph{algebras} to \emph{super}–algebras, the natural ring is the polynomial algebra in
two disjoint sets of variables.
\[
   x=(x_{1},\dots ,x_{k}),\qquad
   y=(y_{1},\dots ,y_{\ell}),
\]
and the correct symmetry notion is \emph{supersym\-metry}: simultaneous invariance under permutations of the $x$’s \emph{and} of the $y$’s together with the \emph{cancellation property} $f(x_{k}=t,\;y_{\ell}=-t)=f$.

For example, supersymmetric Schur functions $S_{\lambda}(x,y)$ which can be seen as characters of $\mathfrak{gl}(k\,|\,\ell)$ introduced by Berele–Regev \cite{BR87}, sit at the crossroads of Lie super-algebras and algebraic combinatorics. Yet, \emph{nothing} was known about their Newton polytopes or support geometry. In this paper, we show that they also admit saturated Newton polytopes (see theorem \ref{Main}).  

\subsection{Beyond Rado's technique}
One can find an expansion of $S_\lambda(x,y)$ in terms of classical Schur polynomials, that is
\begin{equation}\label{SchurExpansion}
    S_{\lambda}(x,y) = \sum_{\mu \subseteq \lambda} s_{\mu}(x)s_{(\lambda/\mu)'}(y)
\end{equation}
where $\mu:= (\mu_1,\mu_2,\ldots) \subseteq \lambda:= (\lambda_1,\lambda_2,\ldots)$ if $\sum_1^n \mu_i \leq \sum_1^n \lambda_i$ for all $n \geq 1$. For any partition $\mu=(\mu_1\ge\dots\ge\mu_k\ge0)$ define its
\emph{permutahedron}
\[
   \Pi(\mu)=\Bigl\{\alpha= (\alpha_1,\alpha_2,\ldots,\alpha_k)\in\R_{\ge0}^{\,k}\;\Big|\;
      |\alpha|=|\mu|,
      \;\alpha_1+\dots+\alpha_i\le\mu_1+\dots+\mu_i\;(1\le i<k)\Bigr\},
\]
By Rado's theorem (see \cite{MTY19})
\begin{align}\label{eq:snp_schur}
    \Newton\bigl(s_\mu(x_1,\dots,x_k)\bigr)=\Pi(\mu),
   \qquad
   \Pi(\mu)\cap\Z^{k}=\Supp\bigl(s_\mu\bigr)
\end{align}
so every ordinary Schur polynomial has SNP. From that, it is not hard to prove that each summand in \ref{SchurExpansion} is SNP. We set
\begin{equation} \label{eq:p_mu}
   P_\mu := \Pi(\mu) \times \Pi\bigl((\lambda/\mu)'\bigr) \subseteq\R^{k+\ell}.
\end{equation}
Then we have
\begin{equation} \label{eq:supp_prod}
   \Supp\bigl(s_\mu(x)s_{(\lambda/\mu)'}(y)\bigr) = P_\mu\cap\Z^{k+\ell}. \qedhere
\end{equation}

By \eqref{SchurExpansion} and the non-negativity of all coefficients of the supersymmetric Schur polynomial, we have
\begin{equation} \label{eq:supp_s_lambda}
   \Supp(S_\lambda) = \bigcup_{\mu\subseteq\lambda} \Supp\bigl(s_\mu(x)s_{(\lambda/\mu)'}(y)\bigr) \stackrel{\eqref{eq:supp_prod}}{=} \bigcup_{\mu\subseteq\lambda}(P_\mu\cap\Z^{k+\ell}).
\end{equation}
Consequently, the Newton polytope is
\begin{equation} \label{eq:newton_s_lambda}
   \Newton(S_\lambda) = \Conv\!\left(\bigcup_{\mu\subseteq\lambda} P_\mu\right).
\end{equation}
Unlike the classical situation, the product-polytopes $P_\mu=\Pi(\mu)\times\Pi\!\bigl((\lambda/\mu)'\bigr)\subset\Bbb R^{k+\ell}$ are not ordered by inclusion when we vary $\mu\subseteq\lambda$. Hence a direct analogue of Rado’s theorem ($\mu\subseteq\lambda$  if and only if $P_\mu\subseteq P_\lambda$) can fail.

\subsection{Our method}
Our strategy is to recast the tableau-theoretic definition of a supersymmetric polynomial as a system of linear inequalities $B_\lambda$ whose integer solutions are exactly the exponent vectors of the monomials that appear. Using the mixed Berele–Regev insertion, we show that a content
\((\mathbf a,\mathbf b)\) appears in a \((k,\ell)\)-semistandard tableau of shape \(\lambda\) if and only if it satisfies the hook inequalities together with the size constraint. This yields the exact, purely linear description:
\[
   \Supp(S_\lambda)=B_\lambda,
\]
as proved in Theorem~\ref{thm:support-hook}. In this way, the combinatorics of tableaux is \emph{linearized} into a system of initial-segment inequalities.

We encode those inequalities as a polyhedron:
\[
   H=\{u\in\mathbb R^{k+\ell}\mid \widetilde A u\le \widetilde b\},
\qquad
H\cap\mathbb Z^{k+\ell}=B_\lambda,
\]
as described in Section~\ref{subsec:poly-model}. After a row sign normalization (Proposition~\ref{prop:rowsign}) the constraint matrix becomes an interval (consecutive–ones) matrix, hence totally unimodular (TU)
(Theorem~\ref{thm:TU-tildeA}). By the Hoffman–Kruskal criterion (Theorem~\ref{thm:HK}) the polyhedron \(H\) is integral (Corollary~\ref{thm:integral-H}), thus completing the proof of the SNP property (Theorem~\ref{Main}). As a consequence, interpreting $S_\lambda$ in the Berele–Regev framework, the exponent vectors $(a,b)$ are precisely the weights that occur (with multiplicity given by tableau counts); hence the weight set equals the integer points of the hook polytope $H$, with no gaps.

The first work using total unimodularity method in the study of SNP appeared in \cite{ARY21}. Besides that, a related framework discussed in \cite{PZ23}, this work provides an a framework that simultaneously addresses two intertwined alphabets and prove the SNP of supersymmetric Schur polynomial. Our approach integrates the combinatorial structure and representation‑theoretic properties of supersymmetric Schur polynomials through techniques from integer programming. We believe that this same framework can be extended given further analysis of the underlying combinatorial models to other supersymmetric polynomials, such as the super–Stanley symmetric functions of Fomin–Kirillov \cite{FK96} and the supersymmetric Macdonald polynomials \cite{BFDLM12a}.



\subsection*{Organization}\label{subsec:organization}

Our work contains two parts. In Section~\ref{sec:hook}, we collect preliminaries and prove the \emph{hook–inequality description} of the support of \(S_{\lambda}(x,y)\) via the mixed Robinson--Schensted correspondence. Section~\ref{sec:TU} encodes these inequalities as a polyhedron \(H\subset\R^{k+\ell}\) and shows that its defining matrix is TU, thus completing the proof.

\subsection*{Acknowledge} 
The authors gratefully acknowledge Professor Phung Ho Hai for his insightful suggestions and valuable encouragement during the course of this research. We would like to thank Professor Alexander Yong for drawing our attention to his paper \cite{ARY21} and for providing valuable insights regarding the TU method. The second author also gratefully acknowledges the generous support and excellent facilities provided by the Morningside Center of Mathematics, Chinese Academy of Sciences. This project was initiated at the “Arithmetic–Algebraic Geometry” conference held in Nha Trang, Khanh Hoa, and we thank the organizers and the University of Khanh Hoa for their kind invitation and warm hospitality.

\section{Supersymmetric Schur polynomials}\label{sec:hook}

We recall Berele--Regev's setting in \cite[Section 2]{BR87}. Fix integers $k,\ell\ge 1$ and declare the ordered “super–alphabet”
\[
  t_1 < t_2 < \dots < t_k < u_1 < u_2 < \dots < u_\ell.
\]
A partition $\lambda$ lies in the \emph{hook}
\[
  H(k,\ell)\;:=\;\bigl\{\lambda \mid \lambda_{k+1}\le \ell \bigr\}.
\]
The main ingredient in Berele--Regev's theory is the notion $(k,\ell)$–semistandard tableau given as follows.
\begin{definition}
    Let $D_\lambda$ be the Young diagram of $\lambda$. A \emph{$(k,\ell)$–semistandard tableau} of shape $\lambda$, abbreviated by $\text{SSYT}_{k,l}(\lambda)$, is a filling of $D_\lambda$ with the letters $t_i,u_j$ such that
\begin{itemize}
  \item $t$–letters $t_1,\dots ,t_k$ are weakly increasing \emph{along rows} and strictly increasing \emph{down columns};
  \item $u$–letters $u_1,\dots ,u_\ell$ are weakly increasing \emph{down columns} and strictly increasing \emph{along rows}.
\end{itemize}
\end{definition}

For a semistandard tableau $T$ set
\[
  a_i(T):=\#\bigl\{t_i\text{ in }T\bigr\},\qquad
  b_j(T):=\#\bigl\{u_j\text{ in }T\bigr\},
\]
and write the \emph{content}
$(\mathbf a(T),\mathbf b(T))\in\Bbb N^{k+\ell}$. Clearly $\sum_i a_i + \sum_j b_j = |\lambda|$.

\begin{example}
    A $(3,2)$-semistandard tableau of shape $(3,3,1)$
    \begin{figure}[H]
 \centering
  \includegraphics[width=0.3\textwidth]{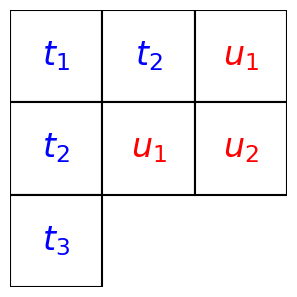}.
 \end{figure}

The content of the tableau
\[
  (\mathbf{a}(T_1), \mathbf{b}(T_1)) = (1, 2, 1, 2, 1).
\]
\end{example}
\begin{example}
A $(2,2)$-semistandard tableau of shape $(4,3,1)$
\begin{figure}[H]
 \centering
  \includegraphics[width=0.3\textwidth]{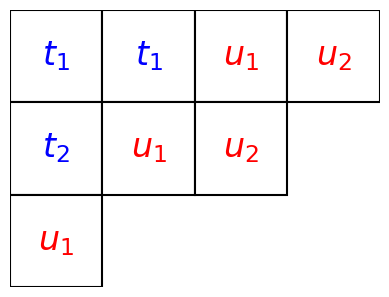}.
 \end{figure}

The content of the tableau
\[
  (\mathbf{a}(T_2), \mathbf{b}(T_2)) = (2, 1, 3, 2).
\]
\end{example}

\textbf{Observation}: For every $1\le r\le k$, because the $t$–letters increase \emph{strictly} \textit{down} every column, all occurrences of $t_1,\dots ,t_r$ are therefore confined to the first $r$ rows, which contain exactly $\lambda_1+\dots +\lambda_r$ boxes; this forces $\sum_{i=1}^r a_i(T)\le \sum_{i=1}^{r}\lambda_i$. The column statement is the left–right mirror image: every $u_1,\dots ,u_s$ lies inside the first $s$ columns. Building on these observations, we introduce the following definition.

\begin{definition}
Let $\lambda\in H(k,\ell)$ be a $(k,\ell)$–hook partition with rows $\lambda=(\lambda_1\ge\lambda_2\ge\cdots)$ and columns $\lambda'=(\lambda'_1\ge\lambda'_2\ge\cdots)$.
For vectors $\mathbf a=(a_1,\dots ,a_k)\in\N^{k}$
and $\mathbf b=(b_1,\dots ,b_\ell)\in\N^{\ell}$ define
\[
  A_{\le r}:=\sum_{i=1}^{r}a_i \quad(1\le r\le k) \quad ,
  \quad
  B_{\le s}:=\sum_{j=1}^{s}b_j \quad(1\le s\le\ell).
\]
We say $(\mathbf a,\mathbf b)$ \emph{satisfies the hook inequalities for $\lambda$} if
\[
  A_{\le r}\le\sum_{i=1}^{r}\lambda_i \quad(1\le r\le k)\quad ,\quad
  B_{\le s}\le\sum_{j=1}^{s}\lambda'_j \quad(1\le s\le\ell)\quad ,\quad
  |\mathbf a|+|\mathbf b|=|\lambda|.
\]
\end{definition}

These discussions can be summarized as follows. 

\begin{lemma}
Let $\lambda\in H(k,\ell)$ and $(\mathbf a,\mathbf b)\in\Bbb N^{k+\ell}$. A $(k,\ell)$–semistandard tableau $T$ of shape $\lambda$ with content $(\mathbf a,\mathbf b)$ satisfies the hook inequalities.
\end{lemma}

We show that the converse of the above lemma is also true, via Berele–Regev’s \emph{$(k,\ell)$-mixed Robinson–Schensted correspondence} (\cite[2.5--2.7]{BR87}).

\begin{theorem}\label{thm:BR-RS}
Let $\mathcal{A}_{k,\ell} = \{t_1 < \dots < t_k < u_1 < \dots < u_\ell\}$ be a super-alphabet. The mixed Robinson--Schensted (RS) correspondence is a bijection that maps a word $w$ built from $\mathcal{A}_{k,\ell}$ to a pair of tableaux $(P(w), Q(w))$ of the same shape $\lambda \in H(k,\ell)$.
\begin{enumerate}
    \item The tableau $P(w)$, called the \emph{insertion tableau}, is a $(k,\ell)$-semistandard tableau whose content is determined by the letter multiplicities in $w$.
    \item The tableau $Q(w)$, called the \emph{recording tableau}, is a standard Young tableau of the same shape $\lambda$, whose entries record the positions at which new cells were added to the shape during the insertion process.
\end{enumerate}
\end{theorem}

We derive the correspondence between $(k,l)$-semistandard tableaux and pairs satisfying hook inequalities.
\begin{proposition}\label{prop:hook-content}
    Let $\lambda\in H(k,\ell)$ and
$(\mathbf a,\mathbf b)\in\Bbb N^{k+\ell}$.
There exists a $(k,\ell)$–semistandard tableau $T$ of shape $\lambda$
with content $(\mathbf a,\mathbf b)$ if and only if
$(\mathbf a,\mathbf b)$ satisfies the hook inequalities.
\end{proposition}

\begin{proof}
    Assume the hook inequalities. Form
\[
   w:=t_1^{a_1}\,t_2^{a_2}\cdots t_k^{a_k}\;
      u_\ell^{\,b_\ell}\cdots u_1^{\,b_1},
\]
the “all $t$’s first, all $u$’s last” word used in their mixed RS correspondence. Insert each letter with the mixed RS rules: row insertion for $t$'s, column insertion for $u$'s. This produces a tableau $P(w)$ that is automatically $(k,\ell)$–semistandard with content $(\mathbf a,\mathbf b)$.

Definition of $w$ gives \(\sum_{i\le r}\operatorname{shape}(P(w))_i=A_{\le r}\)
and \(\sum_{j\le s}\operatorname{shape}(P(w))'_j=B_{\le s}\). Because the hook inequalities say those partial sums do not exceed the corresponding partial sums of $\lambda$, the shape $\mu:=\operatorname{shape}(P(w))$ is dominated by $\lambda$ both row-wise and column-wise, and \(|\mu|=|\lambda|\). The standard dominance lemma (row and column dominance plus equal size implies equality) now forces $\mu=\lambda$. Thus $P(w)$ is a $(k,\ell)$–semistandard tableau of shape $\lambda$ and content $(\mathbf a,\mathbf b)$.
\end{proof}

\begin{definition}[Supersymmetric Schur polynomial]
\label{def:S-lambda}
For commuting indeterminates
\[
   x=(x_1,\dots ,x_k) \quad ,\quad
   y=(y_1,\dots ,y_\ell),
\]
and let $\lambda\in H(k,\ell)$. The \emph{supersymmetric Schur polynomial} of shape $\lambda$ is
\[
   S_\lambda(x,y)
      :=\sum_{T\in\operatorname{SSYT}_{k,\ell}(\lambda)}
          x^{\mathbf a(T)}\,y^{\mathbf b(T)},
\]
where for a tableau $T$ and 
$
   x^{\mathbf a(T)}
      :=x_1^{a_1(T)}\!\cdots x_k^{a_k(T)},
   \;
   y^{\mathbf b(T)}
      :=y_1^{b_1(T)}\!\cdots y_\ell^{b_\ell(T)}.
$
\end{definition}

In \cite{MV03}, Moens and Van der Jeugt gave a determinantal formula for supersymmetric Schur polynomials. The supersymmetric Schur polynomial $S_\lambda(x,y)$ can be computed by the following formula
\[S_{\lambda}(x,y) = \det\left(H_{\lambda_{i}+j-i}(x,y)\right)_{1\leq i,j \leq \ell(\lambda)},\]
where $\ell(\lambda)$ is the number of nonzero components of $\lambda$, and the polynomial 
\[H_r(x,y) = \sum_{i = 0}^rh_i(x)e_{r-i}(y),\]
where $h_i(x)$ is the $i$-th complete homogeneous symmetric polynomial in $k$ variables $x=(x_1,\ldots,x_k)$ and $e_{r-i}(y)$ is the $(r-i)$-th elementary symmetric polynomial in $\ell$ variables $y=(y_1,\ldots,y_\ell)$. It is easy to see that $H_r(x,y)$ is a supersymmetric polynomial, so $S_\lambda(x,y)$ is. Stembridge \cite{Ste89} showed that the set 
\[\{S_\lambda(x,y) \mid \lambda \in H(k,\ell)\}\]
is a linear basis for the ring of supersymmetric polynomials in $x$ and $y$. It was shown that the supersymmetric Schur polynomials are characters of certain representations of the Lie superalgebra $\mathfrak{gl}(m|n)$ (see \cite{Kac77}). A good overview of the supersymmetric Schur polynomials can be found in the PhD thesis of Moens (see \cite{Moe07}).

The first main result of our paper is the following.
\begin{theorem}[Hook description of the support]
\label{thm:support-hook}
One has
\[
  \Supp\bigl(S_\lambda\bigr)
  \;=\;
  \Bigl\{(\mathbf a,\mathbf b)\in\N^{k+\ell}\,\Bigm|\,
        A_{\le r}\le\!\sum_{i\le r}\lambda_i\;(1\!\le\! r\!\le\! k),\;
        B_{\le s}\le\!\sum_{j\le s}\lambda'_j\;(1\!\le\! s\!\le\! \ell),\;
        |\mathbf a|+|\mathbf b|=|\lambda|
  \Bigr\}.
\]
\end{theorem}

\begin{proof}
Expand using Definition~\ref{def:S-lambda}:
\[
  S_\lambda(x,y)
     =\sum_{(\mathbf a,\mathbf b)}
        \bigl|\{T\mid\operatorname{content}(T)=(\mathbf a,\mathbf b)\}\bigr|\;
        x^{\mathbf a}y^{\mathbf b}.
\]
A coefficient is non-zero
if and only if some tableau of shape $\lambda$ has content $(\mathbf a,\mathbf b)$,
and by Proposition~\ref{prop:hook-content} this happens precisely when
$(\mathbf a,\mathbf b)$ satisfies the row, column, and size inequalities
displayed above.
\end{proof}

Put \(\ell=0\) so that the \(y\)-alphabet is empty (equivalently, set
\(y_i=0\) for every \(i\)).  Then \(\operatorname{SSYT}_{k,0}(\lambda)=\operatorname{SSYT}_{k}(\lambda)\) and Definition\ref{def:S-lambda} reduces to the classical Schur polynomial
\[
   s_\lambda(x_1,\dots,x_k)
      = S_\lambda(x,\varnothing)
      = \sum_{T\in\operatorname{SSYT}_{k}(\lambda)}
          x_1^{a_1(T)}\cdots x_k^{a_k(T)},
\]
where \(k\ge\ell(\lambda)\) where \(\ell(\lambda)\) is the length of \(\lambda\). We get a straightforward consequence which resembles the Rado theorem as in \cite{MTY19} for the ordinary Schur polynomial and thus Schur-P and Schur-Q polynomials. 

\begin{corollary}[Rado's theorem]
\label{cor:support-schur}
For any partition \(\lambda\) and any \(k\ge \ell(\lambda)\), one has
\[
  \Supp\bigl(s_\lambda\bigr)
  \;=\;
  \Bigl\{\mathbf a\in\mathbb N^{k}\,\Bigm|\,
        A_{\le r}\;\le\;\sum_{i\le r}\lambda_i
        \;(1\le r\le k),\;
        |\mathbf a|=|\lambda|
  \Bigr\},
\]
where \(A_{\le r}:=\sum_{i\le r}a_i\).
\end{corollary}

\section{Total unimodularity and SNP}\label{sec:TU}
\subsection{Basic definitions}
\label{subsec:TU}
We introduce a linear algebra tool called total unimodularity to prove the integrality of a polyhedron.
\begin{definition}
A matrix $M\in\{0,\pm1\}^{p\times q}$ is
\emph{totally unimodular} (TU) if every square sub–determinant of $M$
is $0,\,+1$ or $-1$.
\end{definition}

The next statement is an elementary but useful observation.

\begin{proposition}[Row–sign invariance]\label{prop:rowsign}
Let $M$ be a matrix and let $D$ be a diagonal matrix with diagonal entries in $\{-1, 1\}$. Then $M$ is totally unimodular (TU) if and only if $DM$ is totally unimodular.
\end{proposition}
\begin{proof}
Let $M' = DM$. Let $M'_{IJ}$ be an arbitrary $k \times k$ submatrix of $M'$ with row indices from the set $I$ and column indices from the set $J$. Let $M_{IJ}$ be the corresponding submatrix of $M$. The operation of pre-multiplying $M$ by the diagonal matrix $D$ scales each row $i$ of $M$ by the corresponding diagonal entry $d_{ii}$. Therefore, the submatrix $M'_{IJ}$ is related to $M_{IJ}$ by the equation
\[ 
M'_{IJ} = D_I M_{IJ}, 
\] 
where $D_I$ is the $k \times k$ diagonal submatrix of $D$ corresponding to the row indices in $I$.

Using the multiplicative property of determinants, we have:
\[ 
\det(M'_{IJ}) = \det(D_I) \det(M_{IJ}). 
\]
Since $D_I$ is a diagonal matrix with entries in $\{-1, 1\}$, its determinant is the product of its diagonal entries, which must be either $-1$ or $1$. That is, $\det(D_I) \in \{-1, 1\}$.

($\Longrightarrow$) Assume $M$ is totally unimodular. By definition, $\det(M_{IJ}) \in \{-1, 0, 1\}$ for any square submatrix $M_{IJ}$. Therefore
\[ 
\det(M'_{IJ}) = (\pm 1) \cdot \det(M_{IJ}) \in \{-1, 0, 1\}. 
\]
Since this holds for any square submatrix of $DM$, the matrix $DM$ is totally unimodular.

($\Longleftarrow$) Assume $DM$ is totally unimodular. Note that $D$ is its own inverse, i.e., $D^{-1} = D$, because $d_{ii}^2 = (\pm 1)^2 = 1$. We can write $M = D^{-1}(DM) = D(DM)$. This means $M$ is obtained by pre-multiplying the totally unimodular matrix $DM$ by a diagonal $\pm 1$ matrix $D$. By the same logic as the forward implication, if $DM$ is TU, then $D(DM) = M$ must also be TU.
\end{proof}

We can also characterize total unimodularity using the following definition.

\begin{definition}
A $\{0,1\}$-matrix is an \emph{interval} (or
\emph{consecutive–ones}) matrix when, after some permutation of columns, the $1$’s in each row appear in a single contiguous block.
\end{definition}

The following is classic.

\begin{theorem}[Heller–Tompkins \cite{HT56};
Fulkerson–Gross \cite{FG65}]
\label{thm:C1_TU}
Every interval matrix is totally unimodular.
\end{theorem}

\begin{definition}
A polyhedron $P\subset\R^{d}$ is \emph{integral} if $P=\operatorname{Conv}(P\cap\Z^{d})$, that is, all its extreme points are lattice points.
\end{definition}

The following result gives an equivalent characterization of integrality and total unimodularity. 
\begin{theorem}[Fulkerson–Gross \cite{FG65}; Hoffman–Kruskal \cite{HK10}]
\label{thm:HK}
For an integer matrix $M\in\Z^{p\times q}$ the following are equivalent:
\begin{enumerate}
\item $M$ is totally unimodular.
\item For every integral $b\in\Z^{p}$ the polyhedron
      $\{x\in\R^{q}\mid Mx\le b\}$ is integral.
\end{enumerate}
\end{theorem}

\subsection{Polyhedral model of supersymmetric Schur polynomials}\label{subsec:poly-model}

\paragraph{\textbf{Setting}}
For integers $k,\ell\ge1$ and $\lambda\in H(k,\ell)=\{\lambda\mid\lambda_{k+1}\le\ell\}$ be a $(k,\ell)$-hook partition, we set
\[
  d:=k+\ell,\qquad
  L_r:=\sum_{i=1}^{r}\lambda_i, \qquad
C_s:=\sum_{j=1}^{s}\lambda'_j
   \qquad(1\le r\le k,\;1\le s\le\ell).
\]

We proved in the previous section that for a vector
$(\mathbf a,\mathbf b)\in\N^{d}\;(=\N^{k}\times\N^{\ell})$
\[
  (\mathbf a,\mathbf b)\in\Supp S_\lambda
  \;\Longleftrightarrow\;
  \bigl(A_{\le r}\le L_r\;(1\le r\le k),\;
        B_{\le s}\le C_s\;(1\le s\le\ell),\;
        |\mathbf a|+|\mathbf b|=|\lambda|\bigr),
\]
where $A_{\le r}:=\sum_{i\le r}a_i$ and $B_{\le s}:=\sum_{j\le s}b_j$. Denote this common set by
\[
   B_\lambda
      :=\{(\mathbf a,\mathbf b)\in\N^{d}\mid
              \text{the three displayed conditions hold}\}.
\]
Our goal is to prove the \emph{saturated Newton polytope} (SNP) property
\begin{equation}\label{eq:SNP}
   \Conv(B_\lambda)\cap\Z^{d}=B_\lambda .
\end{equation}

\paragraph{\textbf{Polyhedral from hook inequalities}}
Write $e:=(1,\dots ,1)\in\R^{d}$ and consider the
$(k+\ell)\times d$ matrix
\[
  A:=\begin{pmatrix}
        \mathbf 1_{1\times1} & 0 & \dots & 0\\
        \mathbf 1_{1\times2} & 0 & \dots & 0\\
        \vdots               &   & \ddots&   \\
        \mathbf 1_{1\times k}& 0 & \dots & 0\\
        0&\mathbf 1_{1\times1} & 0 & \dots \\
        0&\mathbf 1_{1\times2} & 0 & \dots \\
        \vdots&               &\ddots&   \\
        0&\mathbf 1_{1\times\ell}&      &
      \end{pmatrix},
  \quad(\text{$\mathbf 1_{1\times r}$ = row of $r$ ones}).
\]
The first $k$ rows record the inequalities
$A_{\le r}\le L_r$; the last $\ell$ rows record $B_{\le s}\le C_s$.

\begin{lemma}
    $A$ is an interval matrix.
\end{lemma}
\begin{proof}
    Each row of $A$ contains a single block of consecutive $1$’s, so $A$ is an interval matrix and therefore TU by Theorem~\ref{thm:C1_TU}.
\end{proof}

\smallskip
Define the block matrix and vector
\[
   \widetilde A
      :=\begin{pmatrix}
            A\\ e^{\!\top}\\ -e^{\!\top}\\ -I_d
         \end{pmatrix},
   \qquad
   \widetilde b
      :=\begin{pmatrix}
            (L_1,\dots ,L_k,C_1,\dots ,C_\ell)^{\!\top}\\[2pt]
            |\lambda|\\[2pt]
           -|\lambda|\\[2pt]
            0_d
         \end{pmatrix},
\]
and the polyhedron
\[
   H:=\{u\in\R^{d}\mid\widetilde Au\le\widetilde b\}.
\]
By construction, the first $k+\ell$ rows of $\widetilde A$ encode the hook inequalities, the next two rows force the “size’’ equality $|u|=|\lambda|$, the last $d$ rows impose component-wise non-negativity.

\begin{proposition}
    We have
\[
   H\cap\Z^{d}=B_\lambda.
\]

\end{proposition}

\begin{proof}
To see that $H\cap\Bbb Z^d=B_\lambda$ is nothing more than unpacking the definition of $H$, write a generic point

$$
u=(a_1,\dots,a_k,b_1,\dots,b_\ell)\in\R^d.
$$

Then the block–rows of $\widetilde A u\le\widetilde b$ satisfy the following.
\begin{enumerate}
    \item Hook–inequalities:
   For each $1\le r\le k$,

   $$
     \sum_{i=1}^r a_i \;=\;(A_{r,*}\,u)\;\le\;L_r,
   $$

   and for each $1\le s\le\ell$,

   $$
     \sum_{j=1}^s b_j \;=\;(A_{k+s,*}\,u)\;\le\;C_s.
   $$

    \item Size equality: 
    $$
     e^\top u = \sum_{i=1}^k a_i + \sum_{j=1}^\ell b_j \;\le\;|\lambda|,
     \quad
     -\,e^\top u \;\le\;-\,|\lambda|
     \quad\Longrightarrow\quad
     \sum_{i=1}^k a_i + \sum_{j=1}^\ell b_j
     =|\lambda|.
   $$

    \item Nonnegativity: The rows of $-I_d$ force

   $$
     -\,u\;\le\;0
     \quad\Longrightarrow\quad
     a_i\ge0,\;b_j\ge0.
   $$

\end{enumerate}
    Hence
\[
   H\cap\Z^{d}=B_\lambda.
\]
\end{proof}

Thus, proving \eqref{eq:SNP} amounts to showing that every
\emph{integral} point of $H$ is a vertex of
$\Conv(B_\lambda)=H$; equivalently, that the constraint matrix
$\widetilde A$ is totally unimodular.  This will be established in the following theorem.

\begin{theorem}\label{thm:TU-tildeA}
The matrix $\widetilde A
          =\begin{bmatrix} A\\ e^{\!\top}\\ -e^{\!\top}\\ -I_{d}\end{bmatrix}
      \in\{0,\pm1\}^{(k+\ell+d+2)\times d}$
is totally unimodular.
\end{theorem}

\begin{proof}
\emph{Row–sign normalisation.}
Let
$
  D:=\operatorname{diag}\bigl(I_{k+\ell},\,1,\,-1,\,-I_{d}\bigr)
$
and set $\widehat A:=D\widetilde A$.
Multiplying the rows coming from $-e^{\!\top}$ and $-I_d$ by $-1$ turns every non–zero entry of $\widehat A$ into $+1$. By Proposition~\ref{prop:rowsign}, $\widetilde A$ is TU iff $\widehat A$ is TU.

\emph{  $\widehat A$ is interval.}
With the column order inherited from $A$,
$\widehat A$ has the consecutive–ones property:the first $k+\ell$ rows are those of $A$; the next row (all–ones) is trivially interval; each of the final $d$ rows has exactly one $1$. Hence $\widehat A$ is an interval matrix, TU by
Theorem~\ref{thm:C1_TU}, and so $\widetilde A$ is TU.
\end{proof}

\begin{corollary}\label{thm:integral-H}
     $H$ is an integral polyhedron.
\end{corollary}
\begin{proof}
    Since $\widetilde A$ is TU (Theorem~\ref{thm:TU-tildeA}) and $\widetilde b$ is integral, Theorem~\ref{thm:HK} implies that $H$ is an integral polyhedron.
\end{proof}
 
We obtain the main result of our paper.
\begin{theorem}\label{Main}
The supersymmetric Schur polynomial $S_\lambda(x,y)$ has a saturated Newton polytope.
\end{theorem}

\begin{proof}
For any supersymmetric Schur polynomial $S_\lambda$, we have 
\[
   \Newton(S_\lambda)
        =\Conv\bigl(\Supp S_\lambda\bigr)
        =\Conv(B_\lambda)
        =H
        \]
and
\[
   \Newton(S_\lambda)\cap\Z^{d}
        =H\cap\Z^{d}
        =B_\lambda
        =\Supp S_\lambda.
\]

\end{proof}

\begin{remark}
Because $H$ is integral, for every integer vector $c\in\Z^d$, the linear program
\[
\max\{\,c\cdot u \mid \widetilde A\,u \le \widetilde b\}
\]
achieves its maximum at an integral vertex $u^*\in H\cap\Z^d = B_\lambda$ (see \cite{Sch03}).  In other words, optimizing over $H$ by linear programming always returns the exponent of a genuine monomial in $S_\lambda$ that maximizes $c\cdot u$.  Furthermore, because the data $(\widetilde A,\widetilde b)$ are totally unimodular, these linear programs admit strongly polynomial‑time solutions via Tardos’ algorithm \cite{Tar86}.

\end{remark}

\begin{example}
    For $\lambda = (2,1,1)$, we have 
    \[S_{(2,1,1)}(x_1,x_2,y_1) = x_1^2x_2y_1 + x_1x_2^2y_1 + x_1^2y_1^2 + 2x_1x_2y_1^2 + x_2^2y_1^2 + x_1y_1^3 + x_2y_1^3.\]
    The support of $S_{(2,1,1)}(x_1,x_2,y_1)$ is 
    \[\Supp(S_{(2,1,1)}(x_1,x_2,y_1)) = \{(2,1,1),(1,2,1),(2,0,2),(1,1,2),(0,2,2),(1,0,3),(0,1,3)\}.\]
    The Newton polytope of $S_{(2,1,1)}(x_1,x_2,y_1)$, which is the convex hull of $\Supp(S_{(2,1,1)}(x_1,x_2,y_1))$, is a regular hexagon having six vertices 
    \[(2,1,1), (1,2,1), (2,0,2), (0,2,2), (1,0,3), (0,1,3)\]
    and its center $(1,1,2)$ (see Figure \ref{lat_pol}). Thus $S_{(2,1,1)}(x_1,x_2,y_1)$ has SNP.
    \begin{figure}[H]
\centering
\begin{tikzpicture}[scale=1.5]
	\def\tetcol{blue}
	\def\opacity{0.5}
	\tikzstyle{point1}=[ball color=blue, circle, draw=black, inner sep=0.03cm]
	\tikzstyle{point2}=[ball color=red, circle, draw=black, inner sep=0.03cm]
	\tikzstyle{point3}=[ball color=yellow, circle, draw=black, inner sep=0.03cm]
	
	\filldraw[fill=\tetcol!20,rounded corners=0.5pt] (1,0,3) -- (0,1,3) -- (0,2,2) -- (1,2,1) -- (2,1,1) -- (2,0,2) -- cycle;
	\filldraw[fill=\tetcol!20,rounded corners=0.5pt] (0,0,0) -- (1,0,0);
	\filldraw[fill=\tetcol!20,rounded corners=0.5pt] (0,0,0) -- (0,1,0);
	\filldraw[fill=\tetcol!20,rounded corners=0.5pt] (0,0,0) -- (0,0,1);

	\node (A3) at (1,0,3)[point1]{};
	\node (A4) at (0,1,3)[point1]{};
	
	\node (B2) at (2,0,2)[point1]{};
	\node (B3) at (0,2,2)[point1]{};
	
	\node (C1) at (2,1,1)[point1]{};
	\node (C2) at (1,2,1)[point1]{};
	\node (C3) at (1,1,2)[point1]{};
	
	\node (O0) at (0,0,0)[]{};
	\node (O1) at (4,0,0)[]{};
	\node (O2) at (0,4,0)[]{};
	\node (O3) at (0,0,4)[]{};
	
	\node at (A3) [below = 0.5mm]{\tiny$(1,0,3)$};
	\node at (A4) [left = 0.5mm]{\tiny$(0,1,3)$};
	
	\node at (B2) [below = 0.5mm]{\tiny$(2,0,2)$};
	\node at (B3) [above = 0.5mm]{\tiny$(0,2,2)$};
	
	\node at (C1) [right = 0.5mm]{\tiny$(2,1,1)$}; 
	\node at (C2) [above = 0.5mm]{\tiny$(1,2,1)$};
	\node at (C3) [right = 0.5mm]{\tiny$(1,1,2)$};
	
	\node at (O0) [left = 1mm]{\tiny$O$}; 
	\node at (O1) [above = 1mm]{\tiny$e_1$}; 
	\node at (O2) [above = 1mm]{\tiny$e_2$};
	\node at (O3) [above = 1mm]{\tiny$e_3$};
	 
	\draw[-{stealth[scale=1.0]}] (O0) -- (O1);
	\draw[-{stealth[scale=1.0]}] (O0) -- (O2);
	\draw[-{stealth[scale=1.0]}] (O0) -- (O3);
\end{tikzpicture}
\caption{The Newton polytope of $S_{(2,1,1)}(x_1,x_2,y_1)$.} \label{lat_pol}
\end{figure}
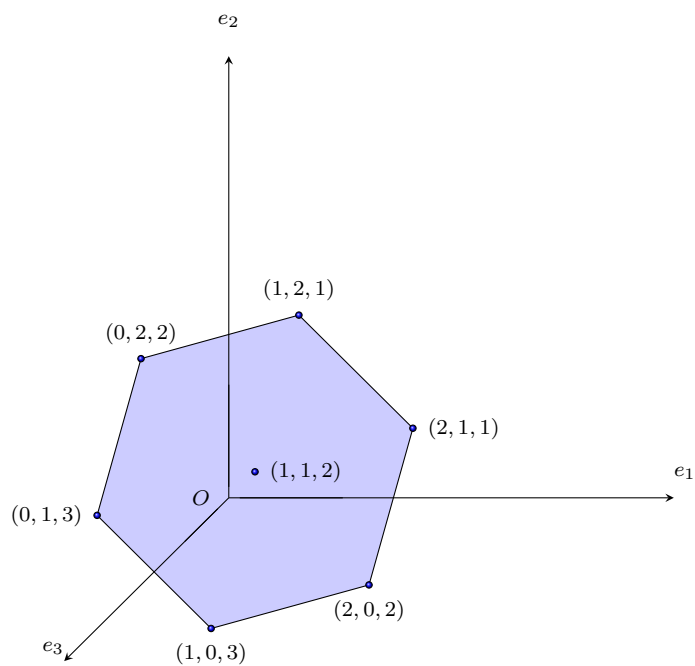
\end{example}

\bibliographystyle{plain}
\bibliography{References}

\end{document}